\documentclass[a4paper,11pt]{amsart}          
\usepackage{amsfonts,amsmath,latexsym,amssymb} 
\usepackage{amsthm}                
\usepackage{mathrsfs,upref}         
\usepackage{mathptmx}		    
\usepackage{float}
\usepackage{tikz}

\newtheorem{theorem}{Theorem}           
\newtheorem{lemma}[theorem]{Lemma}               
\newtheorem{corollary}[theorem]{Corollary}

\theoremstyle{definition}

\newtheorem{definition}[theorem]{Definition}
\newtheorem{example}[theorem]{Example}

\newtheorem{remark}[theorem]{Remark}

\DeclareMathOperator\dist{dist}

\author{Anna Muranova}
 \address{Anna Muranova: IRTG 2235, 
University Bielefeld, Postfach 10 01 31, 
33501 Bielefeld, Germany} 

 \email{anna.muranova@gmail.com}

\title{Effective impedance over ordered fields}
\thanks{This research was supported by IRTG 2235 Bielefeld-Seoul ``Searching for the regular in the irregular:
Analysis of singular and random systems".}

\begin{document}

\maketitle



\begin{abstract}
In this paper we study properties of effective impedance of finite electrical networks and calculate the effective
impedance of a finite ladder network over an ordered field. Moreover, we consider two particular examples of infinite ladder networks (Feynman's network or $LC$-network and $CL$-network, both with zero on infinity) as networks over the ordered Levi-Civita field $\mathcal R$. We show, that effective impedances of finite $LC$-networks converge to the limit in order topology of $\mathcal R$, but the effective impedances of finite $CL$-networks do not converge in the same topology.
\end{abstract}

{\footnotesize
{\bf Keywords:} {weighted graphs, electrical network, ladder network, effective impedance, Laplace operator, ordered field, non-Archimedean field, Levi-Civita field.} 
\smallskip

{\bf Mathematics Subject Classification 2010:}{ 05C22,  34B45, 05C25,  39A12, 12J15.} 
}

\section{Introduction}

It is known that electrical networks with resistances are related to weighted graphs (see e.g. \cite{DS}, \cite{LPW}, \cite{Soardi}). Moreover, it is shown in \cite{Grimmett} and \cite{LPW}, that effective resistance for finite networks satisfies the basic physical properties (e.g. parallel and series laws). In \cite{DS} and \cite{Grimmett} the notion of effective resistance for infinite network is introduced. An effective resistance is tightly related to random walk and Dirichlet problem on graphs, which are described in many papers and books (e. g. \cite{Barlow}, \cite{Grigoryan}, \cite{Woess},\cite{Soardi}).
In \cite{Muranova} a finite electrical network with alternating current and passive elements is considered as a generalization of electrical network with resistances. It is shown there, that such a network is related to weighted graphs over non-Archimedean ordered field of rational functions $\Bbb R(\lambda)$. The generalization of effective resistance for this case is called \emph{effective impedance}. The inverse of effective impedance is called \emph{effective admittance}. The most known in physics infinite network with passive elements is Feynman's ladder network ($LC$-network, see \cite{Feynman}). In \cite{Yoon} the effective impedances of $LC$-network and $CL$-network are considered as limits of complex-valued effective impedances of corresponding sequences of finite networks. 

The present paper consists of two parts. In the first part we describe some properties of electrical network over an ordered field. The main result of this part is Theorem \ref{SM}, which gives the mathematical description of well-known in physics star-mesh transform. The mathematical conceptions of parallel and series laws, as well as $\Delta-Y$ and $Y-\Delta$ transform, follow from Theorem \ref{SM} as corollaries.

In the second part of the present paper we discuss the question whether one can generalize the notion of effective resistance for infinite networks with zero on infinity (see e.g.  \cite{Grimmett}) for the case of non-Archimedean weighted graphs. The main theorem of this section is Theorem \ref{Thmmonotonicity}. It shows, that a sequence of effective admittances of finite networks, exhausted a given infinite network, decreases. Unfortunately, it does not give a convergence over non-Archimedean field. As examples, we consider
$LC$-network and $CL$-network (with zero on infinity) as electrical networks over ordered Levi-Civita field $\mathcal R$, which contains a subfield isomorphic to $\Bbb R(\lambda)$ (see \cite{Berz}, \cite{Hall}). 
Firstly, we present the general calculation of admittance of a finite ladder network (see Figure \ref{figFin}) over an ordered field. Then the closeness of the Levi-Civita field in order topology (\cite{Berz}, \cite{Shams}) gives us an opportunity to arise the question whether effective admittance of infinite network could be defined as limit of effective admittances of corresponding finite electrical networks in this case. We show, that in case of the $LC$-network the sequence of effective admittances of finite networks converge in ordered topology of the Levi-Civita field (Theorem \ref{ThmLC}). Moreover, we show, that admittances of finite $CL$-network do not converge in the same topology (Example \ref{ExCL}).
This shows, that in general it is not possible to generalize the notion of effective resistance for infinite networks for the case of non-Archimedean weights.

\section{Properties of effective impedance of the finite network over ordered field}
\begin{definition}
A network over an ordered field $(K,\succ)$ is a structure
\begin{equation*}
\Gamma=(V,\rho,a_0, B),
\end{equation*}
where 
\begin{itemize}
\item
$(V,E)$ is a locally finite connected graph ($|V|\ge 2$), 
\item
$\rho:E\rightarrow K$ is a positive function called \emph{admittance},
\item
$a_0\in V$ is a fixed vertex,
\item
$B=\{a_1,\dots,a_{|B|}\}\in V\setminus\{a_0\}$ is a fixed non-empty subset of vertices.
\end{itemize}
Let us denote by $B_0=B\cup\{a_0\}$ \emph{the set of boundary vertices} and by $z=\frac{1}{\rho}$ the positive function of \emph{impedance}.
\end{definition}

The network is called \emph{finite} if $|V|<\infty$. Otherwise, it is called \emph{infinite}.

Note that we can consider $\rho$ as a function from $V\times V$ to $K$ by setting $\rho_{xy}=0$, if $xy$ is not an edge.  Then the weight $\rho_{xy}$ gives rise to a function on vertices as follows:
\begin{equation}
\rho(x)=\sum_y \rho_{xy},
\end{equation}
where the notation $\sum\limits_y$ means $\sum\limits_{y\in V}$.
Then $\rho(x)$ is called the \emph{weight of a vertex} $x$. We have $0<\rho(x)<\infty$ for any vertex $x$ of a network .

Let us consider the following \emph{Dirichlet problem} on the given finite network $\Gamma=(V,\rho,a_0, B)$:
\begin{equation}\label{dirpr}
 \begin{cases}
\Delta_\rho v(x)=0 \mbox { on } V\setminus B_0,
   \\
   v(x)= 0 \mbox { on } B,
   \\
 v(a_0)=1,
 \end{cases}
\end{equation}
where $\Delta_\rho v(x)=\sum_y(v(y)-v(x))\rho_{xy}.$

The physical meaning of Dirichlet problem is the following: if we take $K=\Bbb R(\lambda)$ and admittance of each edge in the form
\begin{equation*}
\rho_{xy}=\frac{\lambda}{ L_{xy}\lambda^2+R_{xy}\lambda+{D_{xy}}}, \mbox{ }L_{xy},R_{xy},D_{xy}\ge0 \mbox{ and  } L_{xy}+R_{xy}+D_{xy}\ne 0,
\end{equation*}
then the real voltage at the vertex $x$ at time $t$ will be equal to $\Re(v(x)e^{i\omega t})$, assuming that we keep potential $1$ at the vertex $a_0$, ground all the vertices from $B$ and apply alternating current of frequency $\omega=-i\lambda$ to the network (see \cite{Muranova}).

Note that if $|V|=n$, then the Dirichlet problem \eqref{dirpr} is a $n\times n$ system of linear equations over the field $K$.  It can be also written in a matrix form (note, that here we already have substituted  $v(a_0)=1, v(a_i)=0, i=\overline{1,|B|}$ in the first $(n-|B|-1)$ equations):
\begin{equation}\label{dirprM}
 \begin{cases}
A {\hat v}={b},
   \\
   v(a_0)=1,
   \\
   v(a_i)=0, i=\overline{1,k},
 \end{cases}
\end{equation}
where $k=|B|$, $A$ is a symmetric matrix ($A=A^T$), ${\hat v},{b}$ are vector-columns of length $(n-k-1)$:

\begin{equation*}
A=
\begin{bmatrix}
\sum_{x\sim x_1}\rho_{xx_1}&-\rho_{x_1x_2}&\dots&-\rho_{x_1x_{n-k-1}}\\
-\rho_{x_1x_2}&\sum_{x\sim x_2}\rho_{xx_2}&\dots&-\rho_{x_2x_{n-k-1}}\\
\dots\\
-\rho_{x_1x_{n-k-1}}&-\rho_{x_2x_{n-k-1}}&\dots&\sum_{x\sim x_{n-k-1}}\rho_{xx_{n-k-1}}\\
\end{bmatrix},
\end{equation*}

\begin{equation*}
b=(\rho_{a_0x_1},\rho_{a_0x_2},\dots,\rho_{a_0x_{n-k-1}})^T, 
\end{equation*}
\begin{equation*}
\hat v=(v(x_1),v(x_2),\dots,v(x_{n-k-1}))^T.
\end{equation*}

In \cite{Muranova} it is proved, that the Dirichlet problem \eqref{dirpr} has a unique solution for any finite network over an ordered field.

\begin{definition}
We define \emph{effective impedance} of the finite network $\Gamma$ as
\begin{equation*}
Z_{eff}(\Gamma)=\frac{1}{\sum_{x:x\sim a_0}(1-v(x))\rho_{xa_0}},
\end{equation*}
where $v$ is the solution of the Dirichlet problem \eqref{dirpr}.

The \emph{effective admittance} is defined by
\begin{equation*}
\mathcal P_{eff}{(\Gamma)}=\sum_{x:x\sim a_0}(1-v(x))\rho_{xa_0}.
\end{equation*}
\end{definition}

\begin{lemma}
For the solution $v$ of the Dirichlet problem \eqref{dirpr} we have
\begin{equation}\label{differentEqforP}
\mathcal P_{eff}{(\Gamma)}=\sum_{i=1}^{|B|}\sum_{x:x\sim a_i}v(x)\rho_{xa_i}=\sum_{i=1}^{|B|}\Delta_\rho v(a_i)=-\Delta_\rho v(a_0)=\frac{1}{2}\sum_{\substack{x\sim y\\x,y\in V}}(\nabla_{xy}v)^2\rho_{xy},
\end{equation}
where $\nabla_{xy}v=v(y)-v(x)$.
\end{lemma}
The proof of this result follows the same outline as the proof of the similar result in \cite{Muranova}.

\begin{theorem}{\bf(Star-mesh transform)}\label{SM}
Let $\Gamma = (V,\rho,a_0,B)$  be a finite network, $|V|=n$, $B_0=B\cup\{a_0\}$, and  $x_1,\dots ,x_m\in V$, $3 \le m \le n$, are such that 
\begin{enumerate}
\item
$x_1\not \in B_0$,

\item
$y\not\sim x_1$ for all $y \in V\setminus\{x_2,\dots ,x_m\}$,

\end{enumerate}
If one removes the vertex $x_1$, edges $(x_1,x_i), i=\overline{2,m}$ and change the admittances of the edges $(x_i,x_j), i,j=\overline{2,m}, i\ne j$ as follows:
\begin{equation}\label{rhoSM}
\rho'_{x_ix_j}=\rho_{x_ix_j}+\frac{\rho_{x_1x_i}\rho_{x_1x_j}}{\rho(x_1)},
\end{equation}
not changing the other admittances, then for the new network 
the solution of the Dirichlet problem \eqref{dirpr} for all the vertices will be the same as the solution of the Dirichlet problem \eqref{dirpr} on the original network at corresponding vertices.

\end{theorem}

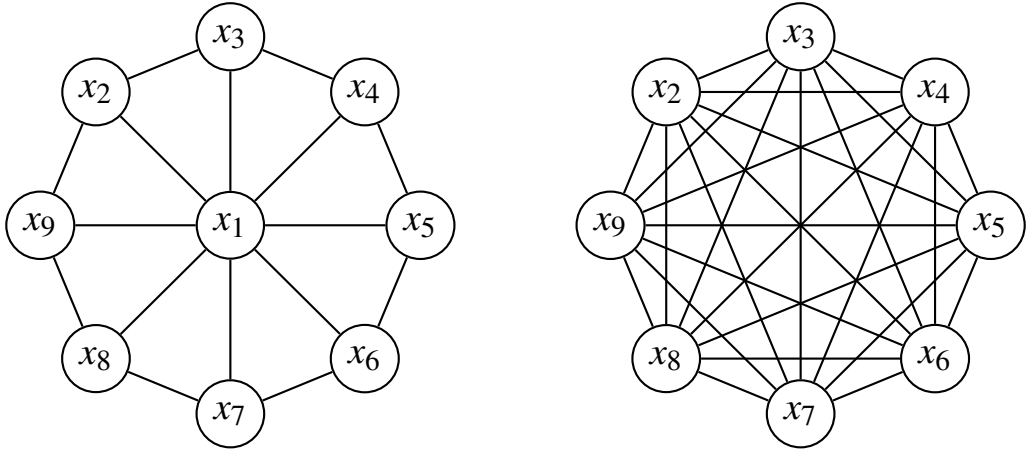
\begin{figure}[H]
\begin{tikzpicture}[auto,node distance=2.5cm,
                    thick,main node/.style={circle,draw,font=\sffamily\Large\bfseries}]
 
  \node[main node] (0) {$x_1$};

  \node[main node] (1) [left of=0] {$x_9$};
  \node[main node] (2) [above left of=0] {$x_2$};
  \node[main node] (3) [above of=0] {$x_3$};
  \node[main node] (4) [above right of=0] {$x_4$};
  \node[main node] (5) [right of=0] {$x_5$};
  \node[main node] (6) [below right of=0] {$x_6$};
  \node[main node] (7) [below of=0] {$x_7$};
  \node[main node] (8) [below left of=0] {$x_8$};

  \path[every node/.style={font=\sffamily\small}]
    (0) edge node [bend right] {} (1)
    (0) edge node [bend right] {} (2)
    (0) edge node [bend right] {} (3)
    (0) edge node [bend right] {} (4)
    (0) edge node [bend right] {} (5)
    (0) edge node [bend right] {} (6)
    (0) edge node [bend right] {} (7)
    (0) edge node [bend right] {} (8)
    (1) edge node [bend right] {} (2)
    (2) edge node [bend right] {} (3)
    (3) edge node [bend right] {} (4)
    (4) edge node [bend right] {} (5)
    (5) edge node [bend right] {} (6)
    (6) edge node [bend right] {} (7)
    (7) edge node [bend right] {} (8)
    (1) edge node [bend right] {} (8);

  \node[main node] (10)[right of=5] {$x_9$};

  \node[main node][draw=none] (00)[right of=10] {};
  \node[main node] (20) [above left of=00] {$x_2$};
  \node[main node] (30) [above of=00] {$x_3$};
  \node[main node] (40) [above right of=00] {$x_4$};
  \node[main node] (50) [right of=00] {$x_5$};
  \node[main node] (60) [below right of=00] {$x_6$};
  \node[main node] (70) [below of=00] {$x_7$};
  \node[main node] (80) [below left of=00] {$x_8$};

  \path[every node/.style={font=\sffamily\small}]

    (10) edge node [bend right] {} (20)
    (10) edge node [bend right] {} (30)
    (10) edge node [bend right] {} (40)
    (10) edge node [bend right] {} (50)
    (10) edge node [bend right] {} (60)
    (10) edge node [bend right] {} (70)
    (10) edge node [bend right] {} (80)
   
    (20) edge node [bend right] {} (30)
    (20) edge node [bend right] {} (40)
    (20) edge node [bend right] {} (50)
    (20) edge node [bend right] {} (60)
    (20) edge node [bend right] {} (70)
    (20) edge node [bend right] {} (80)

    (30) edge node [bend right] {} (40)
    (30) edge node [bend right] {} (50)
    (30) edge node [bend right] {} (60)
    (30) edge node [bend right] {} (70)
    (30) edge node [bend right] {} (80)

    (40) edge node [bend right] {} (50)
    (40) edge node [bend right] {} (60)
    (40) edge node [bend right] {} (70)
    (40) edge node [bend right] {} (80)

    (50) edge node [bend right] {} (60)
    (50) edge node [bend right] {} (70)
    (50) edge node [bend right] {} (80)
    
    (60) edge node [bend right] {} (70)
    (60) edge node [bend right] {} (80)
    (70) edge node [bend right] {} (80);

\end{tikzpicture}
\caption{Star-mesh transform for $m=9$}
\end{figure}

\begin{proof}
Let us consider the Dirichlet problem for the network $\Gamma$ in a matrix form (\ref{dirprM}). Obviously, it is enough to solve the matrix equation $A \hat v=b$.  Without loss of generality we can assume that $x_1,\dots,x_l\not\in B_0$, where $l=m-|\{x_1,\dots,x_m\}\cap B_0|$. Writing equations for $x_1,\dots,x_l$ as the first ones and denoting $k=|B|$, we have
\begin{equation*}
A=
\begin{bmatrix}
\rho(x_1)&-\rho_{x_1x_2}&\dots&-\rho_{x_1x_l}&0&\dots&0\\
-\rho_{x_1x_2}&\rho(x_2)&\dots&-\rho_{x_2x_l}&-\rho_{x_2x_{m+1}}&\dots&-\rho_{x_2x_{n-k-1}}\\
\dots&\dots&\dots&\dots&\dots&\dots&\dots\\
-\rho_{x_1x_l}&-\rho_{x_2x_l}&\dots&\rho(x_l)&-\rho_{x_lx_{m+1}}&\dots&-\rho_{x_lx_{n-k-1}}\\
0&-\rho_{x_2x_{m+1}}&\dots&-\rho_{x_lx_{m+1}}&\rho(x_{m+1})&\dots&-\rho_{x_{m+1}x_{n-k-1}}\\
\dots&\dots&\dots&\dots&\dots&\dots&\dots\\
0&-\rho_{x_2x_{n-k-1}}&\dots&-\rho_{x_lx_{n-k-1}}&-\rho_{x_{m+1}x_{n-k-1}}&\dots&\rho(x_{n-k-1})\\
\end{bmatrix},
\end{equation*}
 since $y\not\sim x_1$ for all $y \in V\setminus\{x_2,\dots ,x_m\}$,
and 
\begin{equation*}
b=(\rho_{a_0x_1},\rho_{a_0x_2},\dots,\rho_{a_0x_l},\rho_{a_0x_{m+1}},\dots,\rho_{a_0x_{n-k-1}})^T,
\end{equation*}
Now it is easy to calculate, that the star-mesh transform is just applying of Gaussian elimination method for the first row.
Indeed, applying Gaussian elimination method for the first row of the augmented matrix $\bar A=[A|b]$ we obtain:
\begin{equation*}
\left[
\begin{array}{ccccccc|c}
1&-\frac{\rho_{x_1x_2}}{\rho(x_1)}&\dots&-\frac{\rho_{x_1x_l}}{\rho(x_1)}&0&\dots&0&\frac{\rho_{a_0x_1}}{\rho(x_1)}\\
0&\rho^*(x_2)&\dots&-\rho'_{x_2x_l}&-\rho_{x_2x_{m+1}}&\dots&-\rho_{x_2x_{n-k-1}}&\rho^*_{a_0x_2}\\
\dots&\dots&\dots&\dots&\dots&\dots&\dots&\dots\\

0&-\rho'_{x_2x_l}&\dots&\rho^*(x_l)&-\rho_{x_lx_{m+1}}&\dots&-\rho_{x_lx_{n-k-1}}&\rho^*_{a_0x_l}\\
0&-\rho_{x_2x_{m+1}}&\dots&-\rho_{x_lx_{m+1}}&\rho(x_{m+1})&\dots&-\rho_{x_{m+1}x_{n-k-1}}&\rho_{a_0x_{m+1}}\\
\dots&\dots&\dots&\dots&\dots&\dots&\dots&\dots\\
0&-\rho_{x_2x_{n-k-1}}&\dots&-\rho_{x_lx_{n-k-1}}&-\rho_{x_{m+1}x_{n-k-1}}&\dots&\rho(x_{n-k-1})&\rho_{a_0x_{n-k-1}}\\
\end{array}
\right]
\end{equation*}

since $\rho(x_1)\ne 0$,
where
\begin{equation*}
\rho^*(x_i)=\rho(x_i)-\frac{\rho_{x_1x_i}^2}{\rho(x_1)} \mbox{ and } \rho^*_{a_0x_i}=\rho_{a_0x_i}+\frac{\rho_{x_1x_i}\rho_{a_0x_1}}{\rho(x_1)} \mbox{ for all } i=\overline{2,l}.
\end{equation*}
Note, that for all $i=\overline{2,l}$
\begin{equation*}
\begin{split}
\rho'(x_i)=&\rho(x_i) -\rho_{x_1x_i}-\sum_{\substack{j=2\\j\ne i}}^m \rho_{x_ix_j}+\sum_{\substack{j=2\\j\ne i}}^m \rho'_{x_ix_j}\\
=&\rho(x_i) -\rho_{x_1x_i}-\sum_{\substack{j=2\\j\ne i}}^m \rho_{x_ix_j}+\sum_{\substack{j=2\\j\ne i}}^m \left(\rho_{x_ix_j}+\frac{\rho_{x_1x_i}\rho_{x_1x_j}}{\rho(x_1)}\right)\\
=&\rho(x_i) -\rho_{x_1x_i}+\sum_{{\substack{j=2\\j\ne i}}}^m \frac{\rho_{x_1x_i}\rho_{x_1x_j}}{\rho(x_1)}\\
=&\rho(x_i) -\rho_{x_1x_i}+\frac{\rho_{x_1x_i}}{\rho(x_1)}\sum_{{\substack{j=2\\j\ne i}}}^m\rho_{x_1x_j} \\
=&\rho(x_i) -\rho_{x_1x_i}+\frac{\rho_{x_1x_i}}{\rho(x_1)}\sum_{j=2}^m\rho_{x_1x_j}-\frac{\rho_{x_1x_i}}{\rho(x_1)}\rho_{x_1x_i}\\
=&\rho(x_i) -\rho_{x_1x_i}+\frac{\rho_{x_1x_i}}{\rho(x_1)}\rho(x_1)-\frac{\rho_{x_1x_i}^2}{\rho(x_1)}\\
=&\rho(x_i)-\frac{\rho_{x_1x_i}^2}{\rho(x_1)}=\rho^*(x_i)
\end{split}
\end{equation*}
and
\begin{equation*}
\begin{split}
\rho^*_{a_0x_i}=\rho_{a_0x_i}+\frac{\rho_{x_1x_i}\rho_{a_0x_1}}{\rho(x_1)}=
\begin{cases}
\rho'_{a_0x_i}, & \text{if}\ a_0 \in\{x_2,\dots,x_m\}\\
\rho_{a_0x_i}, & \text{otherwise, since}\ \rho_{a_0x_1}=0.
 \end{cases}
\end{split}
\end{equation*}
Hence, $\bar A=$
\begin{equation*}
\left[
\begin{array}{ccccccc|c}
1&-\frac{\rho_{x_1x_2}}{\rho(x_1)}&\dots&-\frac{\rho_{x_1x_l}}{\rho(x_1)}&0&\dots&0&\frac{\rho_{a_0x_1}}{\rho(x_1)}\\
0&\rho'(x_2)&\dots&-\rho'_{x_2x_l}&-\rho_{x_2x_{m+1}}&\dots&-\rho_{x_2x_{n-k-1}}&\rho'_{a_0x_2}\\
\dots&\dots&\dots&\dots&\dots&\dots&\dots&\dots\\

0&-\rho'_{x_2x_l}&\dots&\rho'(x_l)&-\rho_{x_lx_{m+1}}&\dots&-\rho_{x_lx_{n-k-1}}&\rho'_{a_0x_l}\\
0&-\rho_{x_2x_{m+1}}&\dots&-\rho_{x_lx_{m+1}}&\rho(x_{m+1})&\dots&-\rho_{x_{m+1}x_{n-k-1}}&\rho_{a_0x_{m+1}}\\
\dots&\dots&\dots&\dots&\dots&\dots&\dots&\dots\\
0&-\rho_{x_2x_{n-k-1}}&\dots&-\rho_{x_lx_{n-k-1}}&-\rho_{x_{m+1}x_{n-k-1}}&\dots&\rho(x_{n-k-1})&\rho_{a_0x_{n-k-1}}\\
\end{array}
\right],
\end{equation*}
Therefore, we can eliminate the variable $v(x_1)$ from the Dirichlet problem, changing admittances as in the statement of the theorem.
\end{proof}

\begin{corollary}\label{CorSM}
Under the star-mesh transform of the network the effective impedance and effective admittance do not change.
\end{corollary}

\begin{proof}
In the proof we will use the notations from the proof of the Theorem \ref{SM}.

The case $\{x_1,\dots, x_m\}\cap B_0=\emptyset$ is trivial. The cases, when $\{x_1,\dots, x_m\}\cap B=\emptyset$ or $\{x_1,\dots, x_m\}\cap \{a_0\}=\emptyset$ are obvious, due to \eqref{differentEqforP}.\\

Otherwise, we can assume, without loss of generality, that 
\begin{equation}
x_{m-j}=a_j,j=\overline{0,|\{x_1,\dots, x_m\}\cap B_0|},
\end{equation}
in particular, $x_m=a_0$.
Then, if we denote the new network by $\Gamma'$ we have by \eqref{differentEqforP}
\begin{align*}
\mathcal P_{eff}(\Gamma)=&-\Delta_\rho v(a_0)=\sum_{x\ne a_0}(1-v(x))\rho_{x a_0}\\
=&(1-v(x_1))\rho_{x_1 a_0}+\sum_{i=2}^{m-1} (1-v(x_i))\rho_{x_ia_0}+\sum_{x\not\in\{x_1,\dots,x_m\}}(1-v(x))\rho_{xa_0}\\
=&\mathcal P_{eff}(\Gamma')-\sum_{i=2}^{m-1} (1-v(x_i))\rho'_{x_ia_0}+(1-v(x_1))\rho_{x_1 a_0}+\sum_{i=2}^{m-1} (1-v(x_i))\rho_{x_ia_0}\\
=&\mathcal P_{eff}(\Gamma')-\sum_{i=2}^{m-1} (1-v(x_i))\frac{\rho_{x_1a_0}\rho_{x_1x_i}}{\rho(x_1)}+(1-v(x_1))\rho_{x_1 a_0}\\
=&\mathcal P_{eff}(\Gamma')-\sum_{i=2}^{m-1} (1-v(x_i))\frac{\rho_{x_1a_0}\rho_{x_1x_i}}{\rho(x_1)}+\left(1-\sum_{i=2}^{m-1}v(x_i)\frac{\rho_{x_1x_i}}{\rho(x_1)}-\frac{\rho_{x_1a_0}}{\rho(x_1)}\right)\rho_{x_1 a_0}\\
=&\mathcal P_{eff}(\Gamma')-\rho_{x_1a_0}\sum_{i=2}^{m-1} (1-v(x_i))\frac{\rho_{x_1x_i}}{\rho(x_1)}+\left(1-\sum_{i=2}^{m-1}v(x_i)\frac{\rho_{x_1x_i}}{\rho(x_1)}-\frac{\rho_{x_1a_0}}{\rho(x_1)}\right)\rho_{x_1 a_0}\\
=&\mathcal P_{eff}(\Gamma')-\rho_{x_1a_0}\left(\sum_{i=2}^{m-1} \frac{\rho_{x_1x_i}}{\rho(x_1)}-\sum_{i=2}^{m-1} v(x_i)\frac{\rho_{x_1x_i}}{\rho(x_1)}\right)+\left(1-\sum_{i=2}^{m-1}v(x_i)\frac{\rho_{x_1x_i}}{\rho(x_1)}-\frac{\rho_{x_1a_0}}{\rho(x_1)}\right)\rho_{x_1 a_0}\\
=&\mathcal P_{eff}(\Gamma')-\rho_{x_1a_0}\left(1-\frac{\rho_{x_1x_m}}{\rho(x_1)}-\sum_{i=2}^{m-1} v(x_i)\frac{\rho_{x_1x_i}}{\rho(x_1)}\right)+\left(1-\sum_{i=2}^{m-1}v(x_i)\frac{\rho_{x_1x_i}}{\rho(x_1)}-\frac{\rho_{x_1a_0}}{\rho(x_1)}\right)\rho_{x_1 a_0}\\
=&\mathcal P_{eff}(\Gamma')
\end{align*}
since 
\begin{equation*}
\rho(x_1)=\sum_{i=2}^m\rho_{x_1x_i}\mbox{ and }v(x_1)=\sum_{i=2}^{l}v(x_i)\frac{\rho_{x_1x_i}}{\rho(x_1)}+\frac{\rho_{x_1a_0}}{\rho(x_1)}=\sum_{i=2}^{m-1}v(x_i)\frac{\rho_{x_1x_i}}{\rho(x_1)}+\frac{\rho_{x_1a_0}}{\rho(x_1)}
\end{equation*}
(see the first line of $\bar A$ and note that $v(x_j)=0$ for all $j=\overline{j+1,m-1}$ and $a_0=x_m$).
\end{proof}

Series law and $Y-\Delta$ transform are just particular cases of star-mesh transform. Since multigraphs are not allowed in this paper, we will use a modification of parallel law and call it parallel-series law.

\begin{corollary}{\bf(Series law)}\label{SeriesLaw}
Let $\Gamma = (V,\rho,a_0,B)$  be a finite network, $B_0=B\cup\{a_0\}$. Let $a,b,c\in V$ are such, that 
\begin{enumerate}
\item
$b\not\in B_0$,
\item
$a\not\sim c$, $a\sim b$, $b\sim c$,
\item
$b\not\sim x$ for all $x\not \in \{a,c\}$.
\end{enumerate}
If one removes the vertex $b$, edges $(a,b),(b,c)$ and add the edge $(a,c)$ with the addmittance
\begin{equation*}
\rho'_{ac}=\frac{\rho_{ab}\rho_{ac}}{\rho_{ab}+\rho_{ac}},
\end{equation*}
not changing other admittances, then for the new network 
the solution of the Dirichlet problem \eqref{dirpr} for all the vertices will be the same as the solution of the Dirichlet problem \eqref{dirpr} on the original network at corresponding vertices. The effective impedance (admittance) of the new network coincides with the effective impedance (addmittance) of the original one.
\end{corollary}
\begin{remark}
The corresponding equation for impedances is then
\begin{equation*}
z'_{ac}={z_{ab}+z_{ac}},
\end{equation*}
which corresponds to the well-known physical series law.
\end{remark}

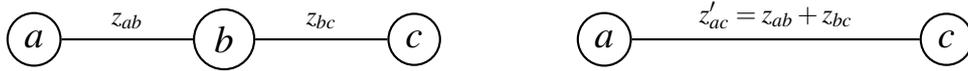
\begin{figure}[H]
\begin{tikzpicture}[auto,node distance=2.5cm,
                    thick,main node/.style={circle,draw,font=\sffamily\Large\bfseries}]

  \node[main node] (1) {$a$};
  \node[main node] (2) [right of=1] {$b$};
  \node[main node] (3) [right of=2] {$c$};

  \path[every node/.style={font=\sffamily\small}]
    (1) edge node [bend right] {$z_{ab}$} (2)
    (2) edge node [bend right] {$z_{bc}$} (3);

  \node[main node] (10) [right of=3]{$a$};
  \node[main node] (20) [right of=10][node distance=4.5cm]  {$c$};

  \path[every node/.style={font=\sffamily\small}]
    (10) edge node [bend right]{$z'_{ac}=z_{ab}+z_{bc}$} (20);
\end{tikzpicture}
\caption{Series law}
\end{figure}

\begin{proof}
Apply Theorem \ref{SM} and Corollary \ref{CorSM} ($x_1=b$) for the case $m=3$ and $\rho_{ac}=0$.
\end{proof}

\begin{corollary}{\bf(Parallel-series law)}\label{SPLaw}
Let $\Gamma = (V,\rho,a_0,B )$  be a finite network, $B_0=B\cup\{a_0\}$.\\ Let $a,b,c\in V$ are such, that 
\begin{enumerate}
\item
$b\not \in B_0$,
\item
$a\sim b,b\sim c, a\sim c$,
\item
$b\not \sim x$ for all $x\not \in \{a,c\}$.
\end{enumerate}
Then if one removes the vertex $b$, edges $(a,b),(b,c)$ and add the edge $(a,c)$ with the admittance
\begin{equation*}
\rho'_{ac}=\frac{\rho_{ab}\rho_{bc}}{\rho_{ab}+\rho_{bc}}+\rho_{ac},
\end{equation*}
not changing other admittances, then for the new network 
the solution of the Dirichlet problem \eqref{dirpr}  for all the vertices will be the same as the solution of the Dirichlet problem \eqref{dirpr} on the original network for corresponding vertices. The effective impedance (admittance) of the new network coincides with the effective impedance (admittance) of the original one.
\end{corollary}

\begin{remark}
The corresponding equation for impedances is then
\begin{equation*}
\frac{1}{z'_{ac}}=\frac{1}{z_{ab}+z_{bc}}+\frac{1}{z_{ac}}
\end{equation*}
which corresponds to the application of the physical series law and then the physical parallel law.
\end{remark}

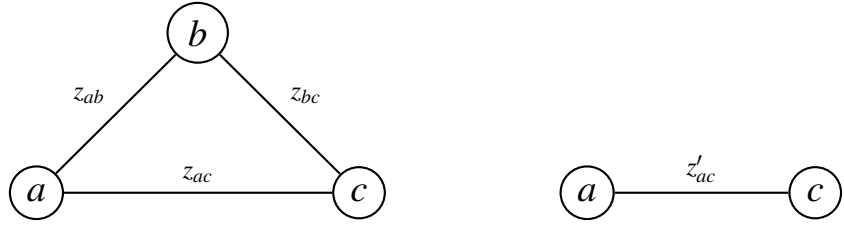
\begin{figure}[H]
\begin{tikzpicture}[auto,node distance=3cm,
                    thick,main node/.style={circle,draw,font=\sffamily\Large\bfseries}]

  \node[main node] (1) {$a$};
  \node[main node] (2) [above right of=1] {$b$};
  \node[main node] (3) [below right of=2] {$c$};

  \path[every node/.style={font=\sffamily\small}]
    (1) edge node [bend right] {$z_{ab}$} (2)
    (2) edge node [bend right] {$z_{bc}$} (3)
    (1) edge node [bend right] {$z_{ac}$} (3);

  \node[main node] (10) [right of=3] {$a$};
  \node[main node] (20) [right of=10] {$c$};

  \path[every node/.style={font=\sffamily\small}]
    (10) edge node [bend right] {$z'_{ac}$} (20);
\end{tikzpicture}
\caption{Parallel-series law}
\end{figure}

\begin{proof}
Apply Theorem \ref{SM} and Corollary \ref{CorSM} ($x_1=b$) for the case $m=3$.
\end{proof}

\begin{theorem}{\bf($Y-\Delta$ transform)}\label{YDLaw}
Let $\Gamma = (V,\rho,a_0,B)$  be a finite network, $B_0=B\cup\{a_0\}$. Let $a,b,c,d\in V$ are such, that 
\begin{enumerate}
\item 
$d\not \in B_0$,
\item
$d \sim a,d\sim b,d\sim c$,
\item
$d\not \sim x$ for all $x\not \in \{a,b,c\}$.
\end{enumerate}

If one removes the vertex $d$, edges $(d,a),(d,b),(d,c)$ and set 
\begin{equation}\label{YD}
\begin{split}
{\rho'_{ab}}=&\frac{\rho_{da}\rho_{db}}{\rho_{da}+\rho_{db}+\rho_{dc}}+\rho_{ab},\\
{\rho'_{bc}}=&\frac{\rho_{db}\rho_{dc}}{\rho_{da}+\rho_{db}+\rho_{dc}}+\rho_{bc},
\\
{\rho'_{ac}}=&\frac{\rho_{da}\rho_{dc}}{\rho_{da}+\rho_{db}+\rho_{dc}}+\rho_{ac},
\end{split}
\end{equation}
not changing other admittances, then for the new network 
the solution of the Dirichlet problem \eqref{dirpr} for all the vertices will be the same as the solution of the Dirichlet problem \eqref{dirpr} on the original network for the corresponding vertices. The effective impedance (admittance) of the new network coincides with the effective impedance (admittance) of the original one.
\end{theorem}

\begin{remark}
The corresponding equalities for the impedances are
\begin{equation}
\begin{split}
{z'_{ab}}=&{\frac{z_{dc}}{z_{da}z_{db}+z_{db}z_{dc}+z_{da}z_{dc}}+\frac{1}{z_{ab}}},\\
{z'_{bc}}=&{\frac{z_{da}}{z_{da}z_{db}+z_{db}z_{dc}+z_{da}z_{dc}}+\frac{1}{z_{bc}}},
\\
{z'_{ac}}=&{\frac{z_{db}}{z_{da}z_{db}+z_{db}z_{dc}+z_{da}z_{dc}}+\frac{1}{z_{ac}}}.
\end{split}
\end{equation}
From the physical point of view, if $\rho_{ab},\rho_{bc},\rho_{ac}$ are all equal to zero, then it is just $Y-\Delta$ transform,
otherwise, it is $Y-\Delta$ transform and the parallel law.
\end{remark}

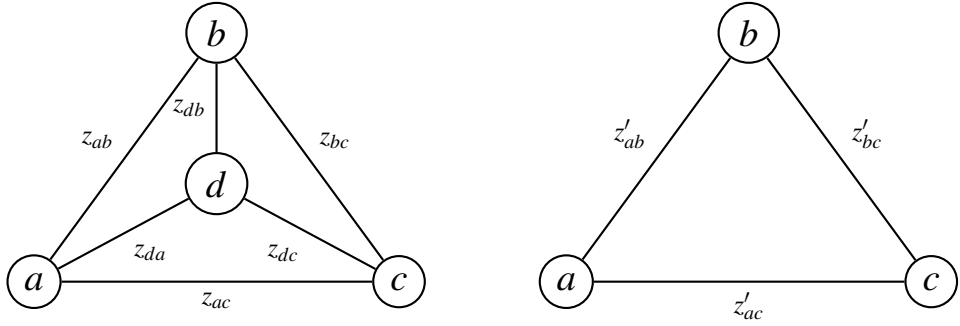
\begin{figure}[H]
\begin{tikzpicture}[auto,node distance=4cm,
                    thick,main node/.style={circle,draw,font=\sffamily\Large\bfseries}]

  \node[main node] at (-0.4,-3.3) (1) {$a$} ;
  \node[main node] at (2,0) (2){$b$};
  \node[main node] at (4.4,-3.3) (3) {$c$};
  \node[main node] at (2,-2) (4) {$d$};

  \path[every node/.style={font=\sffamily\small}]
    (4) edge node [bend right] {$z_{da}$} (1)
    (4) edge node [bend right] {$z_{db}$} (2)
    (3) edge node [bend right] {$z_{dc}$} (4)
    (1) edge node [bend right] {$z_{ab}$} (2)
    (2) edge node [bend right] {$z_{bc}$} (3)
    (3) edge node [bend right] {$z_{ac}$} (1);

  \node[main node] at (6.6,-3.3) (1) {$a$} ;
  \node[main node] at (9,0) (2){$b$};
  \node[main node] at (11.4,-3.3) (3) {$c$};

  \path[every node/.style={font=\sffamily\small}]
    (1) edge node [bend right] {$z'_{ab}$} (2)
    (2) edge node [bend right] {$z'_{bc}$} (3)
    (3) edge node [bend right] {$z'_{ac}$} (1);

\end{tikzpicture}
\caption{$Y-\Delta$ transform} 
\end{figure}

\begin{proof}
Theorem \ref{SM} and Corollary \ref{CorSM} ($x_1=d$) for the case $m=4$.
\end{proof}

The $Y-\Delta$ transform is invertible. In general, it is not the case for star-mesh transform.

\begin{theorem}{\bf($\Delta-Y$ transform)}\label{DYLaw}
Let $\Gamma' = (V',\rho',a_0,B)$  be a finite network and let $a,b,c\in V$ are such, that 
$a \sim b,b \sim c$, and $a\sim c$.
If one add a vertex $d$ and edges $(d,a),(d,b),(d,c)$ setting

\begin{align*}
\rho_{da}&=\frac{\rho'_{ac}\rho'_{bc}+\rho'_{ac}\rho'_{ab}+\rho'_{ab}\rho'_{bc}}{\rho'_{bc}},
\\
\rho_{db}&=\frac{\rho'_{ac}\rho'_{bc}+\rho'_{ac}\rho'_{ab}+\rho'_{ab}\rho'_{bc}}{\rho'_{ac}},
\\
\rho_{dc}&=\frac{\rho'_{ac}\rho'_{bc}+\rho'_{ac}\rho'_{ab}+\rho'_{ab}\rho'_{bc}}{\rho'_{ab}},
\end{align*}
and remove the edges $(a,b),(b,c),(a,c)$
not changing other admittances, then for the new network 
\begin{equation*}
\Gamma=(V\cup\{d\},\rho,a_0,B),
\end{equation*}
the solution of the Dirichlet problem \eqref{dirpr} for any vertex in $V$ will be the same as the solution of the Dirichlet problem \eqref{dirpr} on the original network  for the corresponding vertex. Moreover, the effective impedance and effective admittance do not change under this transform.
\end{theorem}
\begin{remark}
The corresponding equalities for the impedances are
\begin{align*}\label{zad}
z_{da}&=\frac{z'_{ab}z'_{ac}}{z'_{ab}+z'_{bc}+z'_{ac}},
\\
z_{db}&=\frac{z'_{ab}z'_{bc}}{z'_{ab}+z'_{bc}+z'_{ac}},
\\
z_{dc}&=\frac{z'_{bc}z'_{ac}}{z'_{ab}+z'_{bc}+z'_{ac}}.
\end{align*}
\end{remark}

\begin{figure}[H]
\begin{tikzpicture}[auto,node distance=4cm,
                    thick,main node/.style={circle,draw,font=\sffamily\Large\bfseries}]

  \node[main node] at (-0.4,-3.3) (1) {$a$} ;
  \node[main node] at (2,0) (2){$b$};
  \node[main node] at (4.4,-3.3) (3) {$c$};

  \path[every node/.style={font=\sffamily\small}]
    (1) edge node [bend right] {$z'_{ab}$} (2)
    (2) edge node [bend right] {$z'_{bc}$} (3)
    (3) edge node [bend right] {$z'_{ac}$} (1);

  \node[main node] at (6.6,-3.3) (1) {$a$} ;
  \node[main node] at (9,0) (2){$b$};
  \node[main node] at (11.4,-3.3) (3) {$c$};
  \node[main node] at (9,-2) (4){$d$};

  \path[every node/.style={font=\sffamily\small}]
    (4) edge node [bend right] {$z_{da}$} (1)
    (4) edge node [bend right] {$z_{db}$} (2)
    (3) edge node [bend right] {$z_{dc}$} (4);

\end{tikzpicture}
\caption{$\Delta-Y$ transform} 
\end{figure}
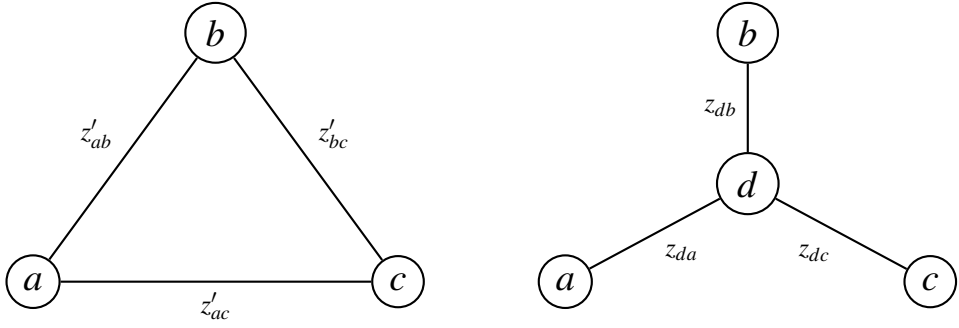

\begin{proof}
To prove the theorem it is enough to express $\rho_{da}$, $\rho_{db}$ and $\rho_{dc}$ from \eqref {YD}, assuming $\rho_{ab}=0$, $\rho_{bc}=0$, and $\rho_{ac}=0$.
Summing up the inverses of all three equations one obtains
\begin{equation*}
\frac{1}{\rho'_{ab}}+\frac{1}{\rho'_{bc}}+\frac{1}{\rho'_{ac}}=\frac{(\rho_{ab}+\rho_{bc}+\rho_{ac})^2}{\rho_{da}\rho_{db}\rho_{dc}}
\end{equation*}
Since both sides are strictly positive, the last equation is equivalent to
\begin{equation}\label{eq1}
\frac{\rho'_{ab}\rho'_{bc}\rho'_{ac}}{\rho'_{ab}\rho'_{bc}+\rho'_{bc}\rho'_{ac}+\rho'_{ab}\rho'_{ac}}=\frac{\rho_{da}\rho_{db}\rho_{dc}}{(\rho_{ab}+\rho_{bc}+\rho_{ac})^2}
\end{equation}
Multiplying the both sides of the last equation by 
\begin{equation*}
\frac{1}{\rho'_{ab}\rho'_{ac}}=\frac{(\rho_{ab}+\rho_{bc}+\rho_{ac})^2}{\rho_{da}^2\rho_{db}\rho_{dc}},
\end{equation*}
which follows from \eqref {YD},
we get
\begin{equation*}
\frac{\rho'_{bc}}{\rho'_{ab}\rho'_{bc}+\rho'_{bc}\rho'_{ac}+\rho'_{ab}\rho'_{ac}}=\frac{1}{\rho_{da}}.
\end{equation*}
Then the equation for $\rho_{da}$ follows.
To obtain the equations for $\rho_{db}$ and $\rho_{dc}$ one should multiply \eqref{eq1} by $\frac{1}{\rho'_{ab}\rho'_{bc}}$ and $\frac{1}{\rho'_{ac}\rho'_{bc}}$ respectively.\\

The fact that effective impedance and effective admittance do not change follows from Theorem \ref{YDLaw}.
\end{proof}
\begin{remark}
All described in this section transforms preserve the positivity of admittances and impedances on the edges.
\end{remark}

\section{Effective impedance of infinite networks over an ordered field}

\subsection{Infinite networks with zero potential on infinity}
Let $\Gamma=(V,\rho,a_0,B)$ be an infinite network over an ordered field $(K,\succ)$. Let us consider the sequence of finite graphs $(V_n, \left.\rho\right|_{V_n})$, where $V_n=\{x\in V\mid\dist (a_0,x)\le n\}$, $n\in\Bbb N$. 

We denote by
\begin{equation*}
\partial V_n=\{x\in V\mid\dist (a_0,x)= n\}
\end{equation*}
 the \emph{boundary} of the graph $(V_n,\left.\rho\right|_{V_n})$. 
Note that $V_{n+1}= \partial V_{n+1}\cup V_n$. 

Let us denote $B_n=B\cap V_n$. Then 
\begin{equation*}
\Gamma_n=(V_n,\left.\rho\right|_{V_n},a_0,B_n\cup \partial V_n),n\in\Bbb N
\end{equation*}
is a \emph{sequence of finite networks exhausted the infinite network $\Gamma$}.

This is an analogue to the approach to infinite networks in \cite{Grimmett}. 

Let us consider the Dirichlet problem \eqref{dirpr} on each $\Gamma_n$: 
\begin{equation}\label{DirprFin}
 \begin{cases}
\sum_{y: y\sim x}(v^{(n)}(y)-v^{(n)}(x))\rho_{xy}=0 \mbox { on } V_n\setminus (\partial V_n\cup B_n\cup\{a_0\}),
\\

   v^{(n)}(x)=0\mbox { on }  \partial V_n\cup B_n.
\\
 v^{(n)}(a_0)=1,

 \end{cases}
\end{equation}

\begin{theorem}\label{Thmmonotonicity}
\begin{equation}\label{monotonicity}
\mathcal P_{eff}(\Gamma_{n+1})\preceq \mathcal P_{eff}(\Gamma_n).
\end{equation}
\end{theorem}

\begin{proof}
By Dirichlet/Thomson’s principle \cite{Muranova} we have
\begin{equation}\label{DTfineq}
\mathcal P_{eff}(\Gamma_{n+1})\preceq \frac{1}{2}\sum_{x,y\in V_{n+1}}(\nabla_{xy} f)^2\rho_{xy}
\end{equation}
 for any $f:V_{n+1}\rightarrow K $ such that $f(a_0)=1, \left.f\right|_{ \partial V_{n+1}\cup B_{n+1}}\equiv 0$. \\
Since $(V_{n+1}\setminus \partial V_{n+1})=V_n$ and $B_{n+1}\cap V_n=B_n$, the inequality \eqref{DTfineq} is true for
\begin{equation*}
f(x)=
\begin{cases}
v^{(n)}(x), \mbox{ if } x\in V_n,\\
0, \mbox{ if } x\in \partial V_{n+1},\\
\end{cases}
\end{equation*}
where $v^{(n)}$ is the solution of \eqref{DirprFin} for $\Gamma_n$. Then
\begin{align*}
\frac{1}{2}\sum_{x,y\in V_{n+1}}(\nabla_{xy} f)^2\rho_{xy}=&\frac{1}{2}\sum_{x,y\in V_{n}}(\nabla_{xy} f)^2\rho_{xy}+\frac{1}{2}\sum_{x,y\in \partial V_{n+1}}(\nabla_{xy} f)^2\rho_{xy}\\
=&\mathcal P_{eff}(\Gamma_n)+0.
\end{align*}
The last equality, together with \eqref{DTfineq}, gives us $\eqref{monotonicity}$.
\end{proof}
\begin{remark}
Even in a Cauchy complete non-Archimedean ordered field inequalities \eqref{monotonicity} for all $n\in\Bbb N$ do not imply, that the sequence $\{\mathcal P_{eff}(\Gamma_n)\}_{n=1}^\infty$ converges. Obviously, if the sequence of effective admittances of finite networks converges, then the corresponding sequence of the effective impedances also has a limit (finite or infinite).
\end{remark}

\begin{definition}
If for given infinite network $\Gamma$ the limit of effective admittances (impedances) of exhausted finite networks exists in $K$, we call it \emph{effective admittance (impedance) of the network $\Gamma$ with zero potential at infinity} and denote it by $\mathcal P_{eff}(\Gamma)$ ($Z_{eff}(\Gamma)$).
\end{definition}

\subsection{Examples: ladder networks over Levi-Civita field}
In this subsection we will investigate the behavior of the sequence $\{\mathcal P_{eff}(\Gamma^{\alpha\beta}_n)\}_{n=1}^\infty$ of effective admittances of finite networks exhausted the ladder network ($\alpha\beta$-network) at the Figure \ref{fig} ($\alpha,\beta \in K$, $\alpha,\beta\succ 0$). More precisely, \emph{$\alpha\beta$-network} is a network $\Gamma^{\alpha\beta}=\{V,\rho, a_0,B\}$, where
\begin{itemize}
\item
$V=\{a_0,a_1,a_2,\dots,x_1,x_2,\dots\}$,
\item 
$\rho_{a_0x_1}=\alpha$, 
$\rho_{x_ix_{i+1}}=\alpha, \rho_{a_ix_i}=\beta, i\in \Bbb N$,
and
$\rho_{xy}=0,\mbox {otherwise}$,
\item
$B=\{a_0,a_1,a_2,\dots\}$.
\end{itemize}
 This network is similar to Feynman's ladder network and $CL$-network (see \cite{Feynman}, \cite{Yoon}), but has zero potential at infinity. Therefore, for any ordered field $K$ the Theorem \ref{Thmmonotonicity} is true for this network. We will show (Theorem \ref{ThmLC} and Example \ref{ExCL}) that whether $\{\mathcal P_{eff}(\Gamma_n)\}_{n=1}^\infty$ converges in Cauchy completness of $K$ depends on $\alpha$ and $\beta$.

\begin{figure}[H]
\begin{tikzpicture}[auto,node distance=2cm,
                    thick,main node/.style={circle,draw,font=\sffamily\small}]
 
  \node[main node] (0) [draw=none]{};
  \node[main node] (1) [right of=0] {$a_1$};
  \node[main node] (2) [right of=1] {$a_2$};
  \node[main node] (3) [right of=2] {$a_k$};
  \node[main node] (4) [right of=3] {$a_{n}$};

  \node[main node] (5) [above of=0]{$a_0$};
  \node[main node] (6) [right of=5] {$x_1$};
  \node[main node] (7) [right of=6] {$x_2$};
  \node[main node] (8) [right of=7] {$x_k$};
  \node[main node] (9) [right of=8] {$x_{n}$};
  \node[main node] (10) [right of=9] [draw=none]{};

  \path[every node/.style={font=\sffamily\small}]
    (5) edge node [bend right] {$\alpha$} (6)
    (6) edge node [bend right] {$\alpha$} (7)
    (9) edge node [bend right] {$\alpha$} (10); 
    \draw[dashed]  (7) to (8);
    \draw[dashed]  (8) to (9);

  \path[every node/.style={font=\sffamily\small}]
    (1) edge node [bend right] {$\beta$} (6) 
    (2) edge node [bend right] {$\beta$} (7)
    (3) edge node [bend right] {$\beta$} (8)
    (4) edge node [bend right] {$\beta$} (9);

\end{tikzpicture}
\caption{$\alpha\beta$-network}\label{fig}
\end{figure}
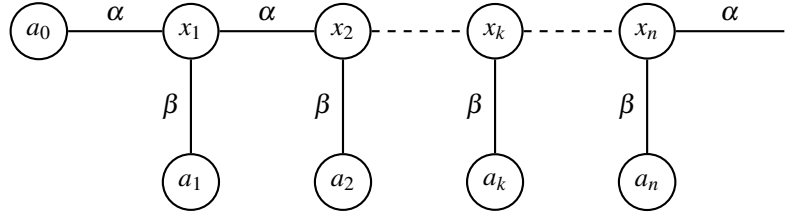

\subsubsection{Finite ladder network over ordered field}
Let $\Gamma_{n}$  be a sequence of finite networks exhausted an $\alpha\beta$-network (see Figure \ref{figFin}).

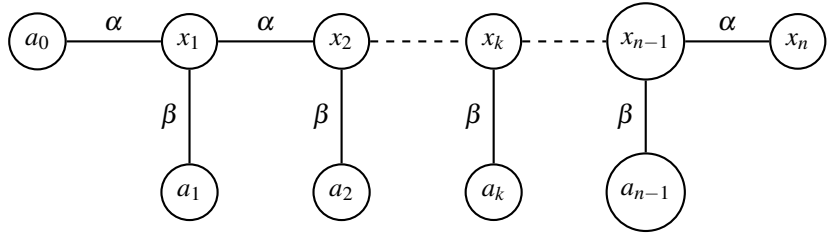
\begin{figure}[H]
\begin{tikzpicture}[auto,node distance=2cm,
                    thick,main node/.style={circle,draw,font=\sffamily\small}]
 
  \node[main node] (0) [draw=none]{};
  \node[main node] (1) [right of=0] {$a_1$};
  \node[main node] (2) [right of=1] {$a_2$};
  \node[main node] (3) [right of=2] {$a_k$};
  \node[main node] (4) [right of=3] {$a_{n-1}$};

  \node[main node] (5) [above of=0]{$a_0$};
  \node[main node] (6) [right of=5] {$x_1$};
  \node[main node] (7) [right of=6] {$x_2$};
  \node[main node] (8) [right of=7] {$x_k$};
  \node[main node] (9) [right of=8] {$x_{n-1}$};
  \node[main node] (10) [right of=9] {$x_{n}$};

  \path[every node/.style={font=\sffamily\small}]
    (5) edge node [bend right] {$\alpha$} (6)
    (6) edge node [bend right] {$\alpha$} (7)
    (9) edge node [bend right] {$\alpha$} (10); 
    \draw[dashed]  (7) to (8);
    \draw[dashed]  (8) to (9);

  \path[every node/.style={font=\sffamily\small}]
    (1) edge node [bend right] {$\beta$} (6) 
    (2) edge node [bend right] {$\beta$} (7)
    (3) edge node [bend right] {$\beta$} (8)
    (4) edge node [bend right] {$\beta$} (9);

\end{tikzpicture}
\caption{Finite ladder network}\label{figFin}
\end{figure}

The Dirichlet problem \eqref{dirpr} for this network is the following
\begin{equation}\label{dirprF}
 \begin{cases}
\alpha v(a_0)+\beta v(a_1)+\alpha v(x_2)-(2\alpha+\beta)v(x_1)=0,
\\
\alpha v(x_{i-1})+\beta v(a_i)+\alpha v(x_{i+1})-(2\alpha+\beta)v(x_i)=0 \mbox { for } i=\overline{2,n-1},
   \\
v(x_n)=0,
\\
   v(a_i)= 0 \mbox { for } i=\overline{1,n-1},
   \\
 v(a_0)=1,
 \end{cases}
\end{equation}

Using the second line in \eqref{dirprF} we obtain the following recurrence relation for $v(x_i), i=\overline{2,n-1}$
\begin{equation}\label{recuk}
v(x_{i+1})-\left(2+\frac{\beta}{\alpha}\right)v(x_i)+v(x_{i-1})=0
\end{equation}
since $v(a_i)=0$ for $i=\overline{2,n-1}$.
The characteristic polynomial of \eqref{recuk} is
\begin{equation}\label{psi}
\psi^2-\left(2+\frac{\beta}{\alpha}\right)\psi+1=0.
\end{equation}
Its roots are 
\begin{equation*}
\psi_{1,2}=1+\frac{\beta}{2\alpha}\pm \xi,
\end{equation*}
where
\begin{equation}\label{xi}
\xi =\sqrt {\frac{\beta}{\alpha}+\left(\frac{\beta}{2\alpha}\right)^2}.
\end{equation}
Note that $\xi$ should not necessary belong to $K$. It is known, that any ordered field posses a real-closed (or maximal ordered) extension $\overline {K}$. Then in $\overline {K}$ exists exactly one positive square root of ${\frac{\beta}{\alpha}+\left(\frac{\beta}{2\alpha}\right)^2}$ (\cite{Bourbaki}). Therefore, we fix the extension $\overline {K}$, denote the positive square root by $\xi$, and make all the further calculations in $\overline {K}$.

The solution of the recurrence equation \eqref{recuk} is
\begin{equation}\label{ukeq}
v(x_i)=c_1\psi_1^i+c_2\psi_2^i,
\end{equation}
where $c_1,c_2\in \overline{K}$ are arbitrary constants.

We use first and third equations in \eqref{dirprF} as boundary conditions for this recurrence equation.
Substituting \eqref{ukeq} in the boundary conditions we obtain the following equations for the constants:
\begin{equation*}\label{const}
\begin{cases}
1+c_1\psi_1^2+c_2\psi_2^2-\left(2+\frac{\beta}{\alpha}\right)(c_1\psi_1+c_2\psi_2)=0,\\
c_1\psi_1^{n}+c_2\psi_2^{n}=0,
\end{cases}
\end{equation*}
which, by \eqref{psi} is equivalent to
\begin{equation*}\label{const}
\begin{cases}
c_1+c_2=1,\\
c_1\psi_1^{n}+c_2\psi_2^{n}=0.
\end{cases}
\end{equation*}
Therefore,
\begin{equation*}\label{const}
\begin{cases}
c_1=\frac{1}{1-\psi_1^{2n}},\\
c_2=\frac{1}{1-\psi_2^{2n}}=\frac{-\psi_1^{2n}}{1-\psi_1^{2n}},
\end{cases}
\end{equation*}
since $\psi_1\psi_2=1$ by \eqref{psi}.

Now we can calculate the effective admittance of $\Gamma_n$:

\begin{equation}\label{PeffF}
\mathcal P_{eff}\left(\Gamma_n\right)={\alpha\left(1-v(x_1)\right)}=\alpha\left(1-c_1 \psi_1-c_2 \psi_2\right)=\frac{\alpha\left(\left(1+\frac{\beta}{2\alpha}+ \xi\right)^{2n-1}+1\right)\left(\frac{\beta}{2\alpha}+\xi\right)}{\left(1+\frac{\beta}{2\alpha}+ \xi\right)^{2n}-1}.
\end{equation}

Since $\mathcal P_{eff}\left(\Gamma_n\right)$ is an element of $K$ as the solution of the Dirichlet problem \eqref{dirpr} over $K$, it can be written without $\xi$:
\begin{align*}
\mathcal P_{eff}\left(\Gamma_n\right)=&\alpha\left(1-c_1 \psi_1-c_2 \psi_2\right)=\alpha\left(1-\frac{\psi_1}{1-\psi_1^{2n}}-\frac{\psi_2}{1-\psi_2^{2n}}\right)\\
=&\alpha\left(1-\frac{\psi_1\left(1-\psi_2^{2n}\right)+\psi_2\left(1-\psi_1^{2n}\right)}{\left(1-\psi_1^{2n}\right)\left(1-\psi_2^{2n}\right)}\right)=\alpha\left(1-\frac{\psi_1+\psi_2-\left(\psi_2^{2n-1}+\psi_1^{2n-1}\right)}{2-\left(\psi_1^{2n}+\psi_2^{2n}\right)}\right)\\
=&\alpha\left(1-\frac{2+\frac{\beta}{\alpha}-2\sum_{k=0}^{n-1}{{2n-1}\choose{2k}}\left(1+\frac{\beta}{2\alpha}\right)^{2n-2k-1}\left(\frac{\beta}{\alpha}+\left(\frac{\beta}{2\alpha}\right)^2\right)^{k}}{2-2\sum_{k=0}^{n}{{2n}\choose{2k}}\left(1+\frac{\beta}{2\alpha}\right)^{2n-2k}\left(\frac{\beta}{\alpha}+\left(\frac{\beta}{2\alpha}\right)^2\right)^{k}}\right),
\end{align*}
where in the last line we have used binomial expansion.

\subsubsection{Infinite ladder networks over Levi-Civita field $\mathcal R$}
We will consider two examples of $\alpha\beta$-network over Levi-Civita field $\mathcal R$. Firstly, let us describe the Levi-Civita field $\mathcal R$. We take the definition of $\mathcal R$ and theorems about its properties from \cite{Berz} and \cite{Shams}.
\begin{definition}\label{LCfield}
A subset $M$ of the rational numbers $\Bbb Q$ is called left-finite if for every $r\in \Bbb Q$ there are only finitely elements of $M$ that are smaller than $r$.

Then the Levi-Civita field is the set of all real valued functions on $\Bbb Q$ with left-finite support with the following operations:
\begin{itemize}
\item
\emph{addition} is defined component-wise
\begin{equation*}
(\alpha+\beta)(q)=\alpha(q)+\beta(q),
\end{equation*}
\item
\emph{multiplication} is defined as follows
\begin{equation*}
(\alpha\cdot \beta)(q)=\sum_{\substack{q_\alpha,q_\beta\in \Bbb Q,\\ q_\alpha+q_\beta=q}} \alpha(q_\alpha)\cdot \beta(q_\beta).
\end{equation*}
\end{itemize}

\end{definition}

It is proved in \cite{Berz}, that $\mathcal R$ is an ordered field with a set of positive elements

\begin{equation}
 \mathcal R^{+}=\{\alpha\in \mathcal R|\alpha(\min \{q\in \Bbb Q\mid \alpha(q)\ne 0\})>0\}.
\end{equation}

We denote by $\tau$ the following element in $\mathcal R$:
\begin{equation}
\tau(q)=
\begin{cases}
1, \mbox{ if } q=1,\\
0, \mbox { otherwise},
\end{cases}
\end{equation}
which plays role of infinitesimal in Levi-Civita field. Therefore, the Levi-Civita field is non-Archimedean.

By \cite{Berz} we can write any $\alpha\in \mathcal R$ as
\begin{equation}\label{LCseries}
\alpha=\sum_{i=1}^\infty \alpha(q_i)\tau^{q_i},
\end{equation}
since $\alpha_n=\sum_{i=1}^n\alpha (q_i)\tau^{q_i}$ converges strongly to the limit $\alpha$ in the order topology.

The set of all polynomials over real numbers $\Bbb R[\tau]=\{a_n \tau^n+\dots a_1 \tau+a_0\mid a_i\in \Bbb R, n\in \Bbb N\}$ is a subring of Levi-Civita field $\mathcal R$ due to \eqref{LCseries}. Therefore, since $\mathcal R$ is a field, the field of rational functions over real numbers
\begin{equation}
\Bbb R(\tau)=\left\{\frac{\sum_{i=k}^n a_i \tau^i}{\sum_{i=l}^m b_i \tau^i}=\frac{a_k\tau^k+a_{k+1}\tau^{k+1}+\cdots+a_n\tau^n}{b_l\tau^l+b_{l+1}\tau^{l+1}+\cdots+b_m\tau^m}\mid a_i,b_i\in\Bbb R, n,m,k,l\in\Bbb N_0\right\}
\end{equation}
is isomorphic to a subfield of $\mathcal R$.




\begin{example}
Let us find the element in $\mathcal R$ which corresponds to the rational function $\frac{1}{\tau^2-4\tau+3}$, i. e. we should find the sequences $\{q_i\}\in \Bbb Q$ and $\{\alpha(q_i)\}\in\Bbb R$ such that

\begin{equation*}
\left(3-4\tau +\tau^2\right)\left(\sum_{q_i}\alpha(q_i)\tau^{q_i} \right)=1.
\end{equation*}
Comparing the coefficients at powers of $\tau$ at right hand side and left hand side, starting from the lowest power, one obtains
\begin{equation*}
q_1=0, \alpha(q_1)=\frac{1}{3}, 
\end{equation*}
\begin{equation*}
q_2=1, \alpha(q_2)=\frac{4}{9}, 
\end{equation*}
and the recurrence relation
\begin{equation*}
q_i=q_{i-1}+1, 3 \alpha(q_i)-4 \alpha(q_{i-1})+\alpha (q_{i-2})=0 \mbox{ for } i>2.
\end{equation*}
Therefore, solving the recurrence relation for $\alpha(q_i)$ we obtain 
\begin{equation*}
\alpha(q_i)=-\frac{1}{2\cdot 3^i}+\frac{1}{2}
\end{equation*}
and
\begin{equation*}
\frac{1}{3-4\tau +\tau^2}=\sum_{i\in \Bbb N}\left(-\frac{1}{2\cdot 3^i}+\frac{1}{2}\right)\tau^{i-1}.
\end{equation*}
\end{example}

Note, that the corresponding order in the field $\Bbb R(\tau)$ is the following:
\begin{equation}\label{OrdR(tau)}
\frac{a_k\tau^k+a_{k+1}\tau^{k+1}+\cdots+a_n\tau^n}{b_l\tau^l+b_{l+1}\tau^{l+1}+\cdots+b_m\tau^m}\succ 0\mbox{ if } \frac{a_k}{b_l}>0.
\end{equation}
Therefore, we can consider Levi-Civita field $\mathcal R$ as an ordered extension of the ordered field $\Bbb R(\tau)$ with the positiveness defined as \eqref{OrdR(tau)}. To consider $\mathcal R$ as an ordered extension of the ordered field $\Bbb R(\tau)$ with the positiveness defined in \cite{Muranova} we make a substitution
\begin{equation}
\tau=\frac{1}{\lambda}.
\end{equation}
 Consequently, we can consider electrical networks over field of rational numbers \cite{Muranova}, as networks over Levi-Civita field and investigate the behavior of the sequence of effective admittances of finite electrical networks. By \cite{Berz} the Levi-Civita field is Cauchy complete in order topology and real-closed.

From the physical point of view we have the following impedances of passive elements 
\begin{itemize}
\item
$(Li\omega)^{-1}=L^{-1}\tau$, $L>0$ for the coil,
\item
$(Ci\omega)=C\tau^{-1}$, $C>0$ for the capacitor,
\item
$R>0$ for the resistor,
\end{itemize}
where $\omega$ is a frequency of the alternating current (see \cite{Muranova}).

Let us consider the Feynman's infinite ladder $LC$-network, assuming that it has zero potential at infinity. It is an $\alpha\beta$-network with $\alpha=L^{-1}\tau$, $\beta={C}\tau^{-1}$, where $L, C>0$, $\alpha,\beta\in \mathcal R$. 

\begin{theorem}\label{ThmLC}
For the Feynman's  ladder $LC$-network ($\alpha=L^{-1}\tau$, $\beta={C}\tau^{-1}$, where $L, C>0$) with zero potential at infinity
\begin{equation}\label{ThmF}
\mathcal P_{eff}\left(\Gamma_n\right)\rightarrow \frac{\beta}{\frac{\beta}{2\alpha}+\xi} \mbox{ as } n\rightarrow\infty
\end{equation}
in the order topology of Levi-Civita field $\mathcal R$, where $\Gamma_n$ is the sequence of the exhausted finite networks
\end{theorem}

\begin{remark}
For the Feynman's ladder $LC$-network
\begin{equation*}
\frac{\beta}{\frac{\beta}{2\alpha}+\xi}=\left(\frac{C}{2\tau}-\frac{\tau}{L}\sqrt{\frac{CL}{\tau^{2}}+\left(\frac{CL}{2 \tau^{2}}\right)^2}\right)
\end{equation*}
and the motivation for this quantity was Feynman's impedance for infinite ladder $LC$-network (see \cite[22-13]{Feynman}).
\end{remark}
\begin{proof}
Firstly, we should write $\xi$ as an element of Levi-Civita field, i. e. as power series \eqref{LCseries}.

\begin{align*}
\xi=&\sqrt{\frac{CL}{\tau^{2}}+\left(\frac{CL}{2 \tau^{2}}\right)^2}=\frac{CL}{2 \tau^{2}}\sqrt{\frac{4 \tau^2}{CL}+1}=\frac{CL}{2 \tau^{2}}\sum_{k=0}^{\infty}\binom{\frac{1}{2}}{k}\left(\frac{4 \tau^{2}}{CL}\right)^k\\
=&\frac{CL}{2\tau^2}\cdot 1+\frac{CL}{2\tau^2}\cdot \frac{1}{2}\cdot\frac{4 \tau^2}{CL}- \frac{1}{8}\cdot\frac{CL}{2\tau^2}\cdot \left(\frac{4 \tau^2}{CL}\right)^2+o\left(\tau^2\right)\\
=&\frac{CL}{2}\tau^{-2}+1-\tau^2+o\left(\tau^2\right).
\end{align*}
Note that here and further $o\left(\tau^m\right)$, where $m\in \Bbb Z$ means $\sum_{k=m+1}^{\infty}a_k \tau^k, a_k\in \Bbb R$. 

Let us calculate the difference $\mathcal P_{eff}\left(\Gamma_n\right)-\frac{\beta}{\frac{\beta}{2\alpha}+ \xi}$.
\begin{align*}
\begin{split}
&\mathcal P_{eff}\left(\Gamma_n\right)-\frac{\beta}{\frac{\beta}{2\alpha}+ \xi}=\frac{\alpha\left(\frac{\beta}{2\alpha}+ \xi\right)\left(\left(1+\frac{\beta}{2\alpha}+\xi\right)^{2n-1}+1\right)}{\left(1+\frac{\beta}{2\alpha}+\xi\right)^{2n}-1}-\frac{\beta}{\frac{\beta}{2\alpha}+ \xi}\\
=&\frac{\alpha\left(\frac{\beta}{2\alpha}+ \xi\right)^2\left(\left(1+\frac{\beta}{2\alpha}+\xi\right)^{2n-1}+1\right)-\beta\left(\left(1+\frac{\beta}{2\alpha}+\xi\right)^{2n}-1\right)}{\left(\left(1+\frac{\beta}{2\alpha}+\xi\right)^{2n}-1\right)\left(\frac{\beta}{2\alpha}+ \xi\right)}
 \end{split}
\end{align*}
The nominator $A$ of the last expression is
\begin{align*}
\begin{split}
A=&\alpha\left(\left(1+\frac{\beta}{2\alpha}+ \xi\right)^{2n-1}+1\right)\left(\frac{\beta}{2\alpha}+\xi\right)^2-\beta\left(\left(1+\frac{\beta}{2\alpha}+ \xi\right)^{2n}-1\right)   \\
=&\alpha\left(D+1\right)\left(\frac{\beta}{2\alpha}+\xi\right)^2-\beta\left(\left(1+\frac{\beta}{2\alpha}+ \xi\right)D-1\right),
 \end{split}
\end{align*}
where $D=\left(1+\frac{\beta}{2\alpha}+ \xi\right)^{2n-1}$.

Since 
\begin{equation*}
\left(\frac{\beta}{2\alpha}+\xi\right)^2=\frac{\beta^2}{4\alpha^2}+\frac{\beta}{\alpha}\xi+\xi^2=\frac{\beta^2}{2\alpha^2}+\frac{\beta}{\alpha}+\frac{\beta}{\alpha}\xi=\frac{\beta}{\alpha}\left(\frac{\beta}{2\alpha}+1+\xi\right)
\end{equation*}
and 
\begin{equation*}
\xi^2=\frac{\beta}{\alpha}+\frac{\beta^2}{4\alpha^2},
\end{equation*}
we have
\begin{align*}
\begin{split}
A=&\alpha\left(D+1\right)\frac{\beta}{\alpha}\left(\frac{\beta}{2\alpha}+1+\xi\right)-\beta\left(\left(1+\frac{\beta}{2\alpha}+ \xi\right)D-1\right)\\
=&D\cdot 0+2\beta+\frac{\beta^2}{2\alpha}+ \beta\xi=2\alpha\left( \frac{\beta}{\alpha}+\frac{\beta^2}{4\alpha^2}+ \frac{\beta}{2\alpha}\xi\right)\\
=&2\alpha\left( \xi^2+ \frac{\beta}{2\alpha}\xi\right)=2\alpha\xi\left( \xi+ \frac{\beta}{2\alpha}\right).
\end{split}
\end{align*}
Therefore, 
\begin{equation}\label{diff}
\mathcal P_{eff}\left(\Gamma_n\right)-\frac{\beta}{\frac{\beta}{2\alpha}+ \xi}=\frac{2\alpha\xi}{\left(1+\frac{\beta}{2\alpha}+ \xi\right)^{2n}-1}.
\end{equation}
The right hand side of the last expression is, obviously, positive in $(\mathcal R, \succ)$, therefore
\begin{align*}
\left|\mathcal P_{eff}\left(\Gamma_n\right)-\frac{\beta}{\frac{\beta}{2\alpha}+ \xi}\right|=&\frac{2\alpha\xi}{\left(\left(1+\frac{\beta}{2\alpha}+ \xi\right)^{2n}-1\right)}\\
=&\frac{2\tau\xi}{L\left(\left(1+\frac{C L}{2\tau^2}+ \xi\right)^{2n}-1\right)}\\
=&\frac{CL\tau^{-1}+2 \tau -2\tau^3+o\left(\tau^3\right)}{L\left(\left(CL\tau^{-2}+2-\frac{1}{CL}\tau^2+o\left(\tau^2\right)\right)^{2n}-1\right)}\\
=&\frac{CL\tau^{-1}+2 \tau-\frac{2}{CL}\tau^3+o\left(\tau^3\right)}{L\left(CL\tau^{-2}\right)^{2n}+o\left(\tau^{-4n-2}\right)}\\
=&\left(C\tau^{-1}+o\left(\tau^{-1}\right)\right)\left(\frac{1}{\left(CL\right)^{2n}}\tau^{4n}+o\left(\tau^{4n+2}\right)\right)\\
=&\frac{C}{\left(CL\right)^{2n}}\tau^{4n-1}+o\left(\tau^{4n-1}\right)\rightarrow 0,
\end{align*}
when $n\rightarrow \infty$.
\end{proof}

\begin{remark}\label{absuf}
From the proof one can see that  \eqref{ThmF} is true for $\alpha\beta$-network whenever for any $\gamma\in \mathcal R$ exists $N_0\in \Bbb N$ such that $n>N_0$ implies $\left(\frac{\beta}{\alpha}\right)^n\succ \gamma$.
\end{remark}

\begin{example}\label{ExCL}
For the $CL$-network ($\alpha={C}\tau^{-1}$, $\beta=L^{-1}\tau$, $L,C>0$) effective admittances of the exhausted finite networks do not converge in the Levi-Civita field $\mathcal R$.

\begin{proof}
In this case $\xi=\frac{\tau}{\sqrt{CL}}+\left(\frac{\tau}{2\sqrt{CL}}\right)^3+o(\tau^3)$.
Let us prove, that $\{\mathcal P (\Gamma_n)\}_{n=1}^\infty$ is not a Cauchy sequence in $\mathcal R$. Indeed
\begin{align*}
\mathcal P &(\Gamma_{n+1})-\mathcal P (\Gamma_n)\\
=&\frac{\alpha\left(\frac{\beta}{2\alpha}+ \xi\right)\left(\left(1+\frac{\beta}{2\alpha}+\xi\right)^{2n+1}+1\right)}{\left(1+\frac{\beta}{2\alpha}+\xi\right)^{2n+2}-1}
-\frac{\alpha\left(\frac{\beta}{2\alpha}+ \xi\right)\left(\left(1+\frac{\beta}{2\alpha}+\xi\right)^{2n-1}+1\right)}{\left(1+\frac{\beta}{2\alpha}+\xi\right)^{2n}-1}\\
=&\alpha\left(\frac{\beta}{2\alpha}+ \xi\right)\left(\frac{\left(1+\frac{\beta}{2\alpha}+\xi\right)^{2n+1}+1}{\left(1+\frac{\beta}{2\alpha}+\xi\right)^{2n+2}-1}
-\frac{\left(1+\frac{\beta}{2\alpha}+\xi\right)^{2n-1}+1}{\left(1+\frac{\beta}{2\alpha}+\xi\right)^{2n}-1}\right)
\end{align*}
Since
\begin{equation*}
\psi_1=1+\frac{\beta}{2\alpha}+\xi=1+\frac{\tau}{\sqrt{CL}}+o\left(\tau^{1}\right),
\end{equation*}
we can rewrite
\begin{align*}
\mathcal P &(\Gamma_{n+1})-\mathcal P (\Gamma_n)=\alpha\left(\frac{\beta}{2\alpha}+ \xi\right)\left(\frac{\psi_1^{2n+1}+1}{\psi_1^{2n+2}-1}-\frac{\psi_1^{2n-1}+1}{\psi_1^{2n}-1}\right)\\
=&\alpha\left(\frac{\beta}{2\alpha}+ \xi\right)\frac{\left(\psi_1^{2n+1}+1\right)\left(\psi_1^{2n}-1\right)-\left(\psi_1^{2n-1}+1\right)\left(\psi_1^{2n+2}-1\right)}{\left(\psi_1^{2n+2}-1\right)\left(\psi_1^{2n}-1\right)}\\
=&C\tau^{-1}\left(\frac{\tau}{\sqrt{CL}}+o\left(\tau^1\right)\right)\frac{\left(\psi_1^{2n+1}+1\right)\left(\psi_1^{2n}-1\right)-\left(\psi_1^{2n-1}+1\right)\left(\psi_1^{2n+2}-1\right)}{\left(\psi_1^{2n+2}-1\right)\left(\psi_1^{2n}-1\right)}
\end{align*}
Substituting
\begin{align*}
&\left(\psi_1^{2n+1}+1\right)\left(\psi_1^{2n}-1\right)-\left(\psi_1^{2n-1}+1\right)\left(\psi_1^{2n+2}-1\right)\\
=&\left(2+\frac{\tau}{\sqrt{CL}}(2n+1)+o\left(\tau^1\right)\right)\left(\frac{\tau}{\sqrt{CL}}(2n)+o\left(\tau^1\right)\right)\\
-&\left(2+\frac{\tau}{\sqrt{CL}}(2n-1)+o\left(\tau^1\right)\right)\left(\frac{\tau}{\sqrt{CL}}(2n+2)+o\left(\tau^1\right)\right)\\
=&-4\frac{\tau}{\sqrt{CL}}+o\left(\tau^1\right)
\end{align*}
and
\begin{align*}
&\left(\psi_1^{2n+2}-1\right)\left(\psi_1^{2n}-1\right)=\left(\frac{\tau}{\sqrt{CL}}(2n+2)+o\left(\tau^1\right)\right)\left(\frac{\tau}{\sqrt{CL}}(2n)+o\left(\tau^1\right)\right)\\
=&\frac{\tau^2}{CL}(4n^2+4n)+o\left(\tau^2\right)
\end{align*}
we obtain
\begin{align*}
\mathcal P &(\Gamma_{n+1})-\mathcal P (\Gamma_n)\\
&=\left(\frac{C}{\sqrt{CL}}+0(\tau^0)\right)\left(-4\frac{\tau}{\sqrt{CL}}+o\left(\tau^1\right)\right)\left(\frac{CL}{4n^2+4n}\tau^{-2}+o\left(\tau^{-2}\right)\right)\\
&=\frac{-4C}{4n^2+n}\tau^{-1}+o\left(\tau^{-1}\right)\succ \tau \mbox{ for any } n\in \Bbb N \mbox { and for any } L,C>0.
\end{align*}
Therefore, $\{\mathcal P (\Gamma_n)\}_{n=1}^\infty$ is not a Cauchy sequence in $\mathcal R$.
\end{proof}

\end{example}

Therefore, the following question: 
\begin{quote}
Under what conditions the effective admittance of infinite network over non-Archimedean field could be defined?
\end{quote}
remains open. Note, that Remark \ref{absuf} gives some sufficient condition for $\alpha\beta$-network.

\section*{Acknowledgement}
The author thanks her scientific advisor, Professor Alexander Grigor'yan, for fruitful discussions on the topic.

\end{document}